\documentclass[12pt]{amsart}
\usepackage{amssymb}
\usepackage{hyperref}

\theoremstyle{plain}
\newtheorem{theorem}{Theorem}[section]

\newtheorem{corollary}[theorem]{Corollary}
\newtheorem{lemma}[theorem]{Lemma}

\theoremstyle{definition}

\newtheorem{remark}[theorem]{Remark}

\newtheorem{algorithm}[theorem]{Algorithm}

\theoremstyle{remark}

\numberwithin{equation}{section}

\newcommand{\N}{\mathbb N}
\newcommand{\Z}{\mathbb Z}

\newcommand{\C}{\mathbb C}

\newcommand{\SO}{\operatorname{SO}}

\newcommand{\Sp}{\operatorname{Sp}}

\newcommand{\sll}{\mathfrak{sl}}
\newcommand{\so}{\mathfrak{so}}
\newcommand{\spp}{\mathfrak{sp}}

\newcommand{\op}{\operatorname}

\newcommand{\diag}{\operatorname{diag}}

\newcommand{\ba}{\backslash}

\newcommand{\norma}[1]{\|{#1}\|_1}
\newcommand{\contador}{\ell}
\newcommand{\tipo}{\mathrm}

\title[Weight multiplicity formulas]{Weight multiplicity formulas for bivariate representations of classical Lie algebras}
\author{Emilio~A.~Lauret, Fiorela Rossi Bertone}
\address{CIEM--FaMAF (CONICET), Universidad Nacional de C\'ordoba, Medina Allende s/n, Ciudad Universitaria, 5000 C\'ordoba, Argentina.}
\email{elauret@famaf.unc.edu.ar, rossib@famaf.unc.edu.ar}
\date{\today}

\begin{document}

\begin{abstract}
A bivariate representation of a complex simple Lie algebra is an irreducible representation having highest weight a combination of the first two fundamental weights.
For a complex classical Lie algebra, we establish an expression for the weight multiplicities of bivariate representations.
\end{abstract}

\maketitle

\section{Introduction}\label{sec:intro}

This article concerns on giving weight multiplicity formulas, continuing the previous authors' article \cite{LR-fundstring}.
In that article, it was determined, for a classical complex Lie algebra $\mathfrak g$, a closed explicit formula for the weight multiplicities of any representation of any $p$-fundamental string.
Such a representation is an irreducible representation of $\mathfrak g$ with highest weight $k\omega_1+\omega_p$ for some non-negative integer $k$.
Here, $\omega_j$ denotes the $j$th fundamental weight associated with the root system of $\mathfrak g$.

The primary goal of the present article is to find an expression for the weight multiplicity of every \emph{bivariate representation} of a classical complex Lie algebra $\mathfrak g$.
A bivariate representation is an irreducible representation with highest weight $a\omega_1+b\omega_2$ for some non-negative integers $a$ and $b$ (cf.\ \cite{Maddox14}).
See \cite[\S1]{LR-fundstring} for references of classical and recent previous results on this problem.

In Section~\ref{sec:notation} we introduce the standard notation used to describe the root system associated to a classical complex Lie algebra  $\mathfrak g$.
In particular, for $\mathfrak g$ of type $\tipo B_n$, $\tipo C_n$ or $\tipo D_n$ and $\mathfrak h$ a fixed Cartan subalgebra of $\mathfrak g$, $\{\varepsilon_1,\dots,\varepsilon_n\}$ denotes the basis of $\mathfrak h^*$ satisfying that the set of simple roots are $\{\varepsilon_1-\varepsilon_{2},\dots, \varepsilon_{n-1}-\varepsilon_n,\varepsilon_n\}$ for type $\tipo B_n$, $\{\varepsilon_1-\varepsilon_{2},\dots, \varepsilon_{n-1}-\varepsilon_n,2\varepsilon_n\}$ for type $\tipo C_n$, and $\{\varepsilon_1-\varepsilon_{2},\dots, \varepsilon_{n-1}-\varepsilon_{n},\varepsilon_{n-1}+\varepsilon_{n}\}$ for type $\tipo D_n$.
According to this notation, bivariate representations have highest weight of the form  $k\varepsilon_1+l\varepsilon_2$ for integers $k\geq l\geq 0$.

The obtained weight multiplicity formulas for types $\tipo B_n$, $\tipo C_n$ and $\tipo D_n$ are in Theorems~\ref{thmBn:mulip_bivar}, \ref{thmCn:mulip_bivar} and \ref{thmDn:mulip_bivar}  respectively.
The expressions involve a sum over partitions of the integer numbers  $\leq l$, so they may not be considered `closed explicit formulas' like in \cite{LR-fundstring}.
An immediate and curious consequence of the formulas is the next result.

\begin{theorem}\label{thm:depending}
	Let $\mathfrak g$ be a classical complex Lie algebra of type $\tipo B_n$, $\tipo C_n$ or $\tipo D_n$. Let $k\geq l\geq 0$ integers and $\mu=\sum_{i=1}^{n} a_i\varepsilon_i$ with $a_i\in\Z$ for all $i$.
	The multiplicity of $\mu$ in the irreducible representation $\pi_{k\varepsilon_1+l\varepsilon_2}$ of $\mathfrak g$ with highest weight ${k\varepsilon_1+l\varepsilon_2}$, denoted by $m_{\pi_{k\varepsilon_1+l\varepsilon_2}}(\mu)$, depends only on
	\begin{equation}
	\norma{\mu}:=\sum_{i=1}^n |a_i|
	\quad\text{and}\quad
	\contador_t(\mu):=\#\{i: 1\leq i\leq n,\, |a_i|=t\}\quad\text{for all  $0\leq t\leq l-1$}.
	\end{equation}
	In other words, if $\mu$ and $\mu'$ satisfy $\norma{\mu}=\norma{\mu'}$ and $\contador_t(\mu) = \contador_t(\mu')$ for all $0\leq t\leq l-1$, then $m_{\pi_{k\varepsilon_1+l\varepsilon_2}}(\mu) = m_{\pi_{k\varepsilon_1+l\varepsilon_2}}(\mu')$.
\end{theorem}

This theorem is analogous to \cite[Cor.~1.1]{LR-fundstring} (see also \cite[Lem.~3.3]{LMR-onenorm}),  which states that the multiplicity of a weight $\mu$ in representations in $p$-fundamental strings depends only on $\norma{\mu}$ and $\contador_0(\mu)$.
Such representations have highest weights of the form $k\omega_1+\omega_p$ for $k\geq0$ and, $1\leq p\leq n-1$ for type $\tipo B_n$, $1\leq p\leq n$ for type $\tipo C_n$, $1\leq p\leq n-2$ for type $\tipo D_n$.

In the best authors' knowledge, the weight multiplicity formulas in Theorems~\ref{thmBn:mulip_bivar}--\ref{thmDn:mulip_bivar} are not in the literature.
Nevertheless, Maddox~\cite{Maddox14} obtained a multiplicity formula for bivariate representations when $\mathfrak g$ is of type $\tipo C_n$.
However, her expression differs significantly from ours.
In particular, Theorem~\ref{thm:depending} does not follow immediately from her formula.
See Remark~\ref{rem:Maddox} for more details.

We compare from a computational point of view, the multiplicity formulas obtained in Theorems~\ref{thmBn:mulip_bivar}--\ref{thmDn:mulip_bivar} with Freudenthal's famous formula (see Subsection~\ref{subsec:comparison}).
We used the open-source mathematical software \texttt{Sage}~\cite{Sage} to do the calculations.
It was evidenced in the computational results shown in Table~\ref{tabla:comparison} that the bivariate algorithm based on Theorems~\ref{thmBn:mulip_bivar}--\ref{thmDn:mulip_bivar} is faster than Freudenthal algorithm for most of small values of $k$ and $l$.
Moreover, for any choice of $k$ and $l$, the same conclusion would hold for $n$ big enough.
It is probably more significant the fact that Theorems~\ref{thmBn:mulip_bivar}--\ref{thmDn:mulip_bivar} return in a speedy way the multiplicity of a single weight.
The situation is very different with Freudenthal's formula since it is defined recursively, and moreover, it has to calculate the multiplicities of many intermediate weights in case the original weight is far away from the highest weight.
Many more related remarks are made in Subsection~\ref{subsec:comparison}.

We have already mentioned that the expressions for the weight multiplicities in Theorems~\ref{thmBn:mulip_bivar}--\ref{thmDn:mulip_bivar} and \ref{thmAn:multip(k,l)} are not closed explicit formulas since they involve a sum over partitions.
However, in some particular cases, one can write down the corresponding partitions obtaining a (long) closed expression.
For instance, the statements for the multiplicity of the weight $\mu=0$ for types $\tipo B_n$, $\tipo C_n$ and $\tipo D_n$ are included in Subsection~\ref{subsec:closedformulas}.
Furthermore, a formula for $\pi_{k\varepsilon_1+2\varepsilon_2}$ in type $\tipo D_n$ is also part of that subsection.

Although weight multiplicity formulas are interesting in themselves, the authors were motivated by their application in spectral geometry (see \cite[\S7]{LR-fundstring}).
In Remark~\ref{rem:applications} we mention possible applications for the weight multiplicity formulas obtained in this article for the determination of the spectra of some natural differential operators on spaces covered by compact symmetric spaces with abelian fundamental groups.

The weight multiplicity formula for $\mathfrak g$ of type $\tipo A_n$ is determined in Section~\ref{sec:typeAn}.
Furthermore, the corresponding expression for the case $l=2$ is stated in Corollary~\ref{corAn:for(l=2)}.
This case, $\mathfrak g$ of type $\tipo A_n$, is much simpler than the previous ones.
The obtained expressions are probably already present in the extensive literature on this area.

The article is organized as follows.
Section~\ref{sec:notation} introduces the necessary notation to read the statements of the primary results.
Section~\ref{sec:typeBCD}, which considers classical Lie algebras of type $\tipo B_n$, $\tipo C_n$ and $\tipo D_n$, is divided in five subsections.
The first one states the weight multiplicity formulas which are proven in the second one.
The computational comparison is made in Subsection~\ref{subsec:comparison}.
Subsection~\ref{subsec:closedformulas} shows closed explicit formulas in particular cases.
Section~\ref{sec:typeBCD} ends with some remarks.
The case when $\mathfrak g$ is of type $\tipo A_n$ is considered in Section~\ref{sec:typeAn}.

\section{Notation}\label{sec:notation}

Throughout this section, $\mathfrak g$ denotes a classical complex Lie algebra of type $\tipo B_n$, $\tipo C_n$ and $\tipo D_n$, namely $\so(2n+1,\C)$, $\spp(n,\C)$, $\so(2n,\C)$ respectively.
We assume $n\geq2$ for types $\tipo B_n$ and $\tipo C_n$, and $n\geq3$ for $\tipo D_n$.
We fix a Cartan subalgebra $\mathfrak h$ of $\mathfrak g$.

Let $\{\varepsilon_{1}, \dots,\varepsilon_{n}\}$ be the standard basis of $\mathfrak h^*$.
Then, the sets of positive roots $\Sigma^+(\mathfrak g, \mathfrak h)$ and the space of integral weights $P(\mathfrak g)$ are respectively given by
\begin{align*}
\{\varepsilon_i\pm\varepsilon_j: i<j\}\cup\{\varepsilon_i\}
&\quad\text{ and }\quad
\{\textstyle \sum_i a_i\varepsilon_i: a_i\in\Z\,\forall i, \text{ or } a_i-\frac12 \in\Z\,\forall i\}
 &&\text{ for $\mathfrak g$ of type $\tipo B_n$,}\\
\{\varepsilon_i\pm \varepsilon_j: i<j \}\cup\{2\varepsilon_i\}
&\quad\text{ and }\quad
\{\textstyle \sum_i a_i\varepsilon_i: a_i\in\Z\,\forall i\}
 &&\text{ for $\mathfrak g$ of type $\tipo C_n$,}\\
\{\varepsilon_i\pm\varepsilon_j: i<j\}
&\quad\text{ and }\quad
\{\textstyle \sum_i a_i\varepsilon_i: a_i\in\Z\,\forall i, \text{ or } a_i-\frac12\in\Z\,\forall i\}
 &&\text{ for $\mathfrak g$ of type $\tipo D_n$.}
\end{align*}
Furthermore, $\sum_i a_i\varepsilon_i\in P(\mathfrak g)$ is dominant if and only if $a_1\geq a_{2}\geq \dots\geq a_{n}\geq0$ for types $\tipo B_n$ and $\tipo C_n$ and $a_1\geq  \dots\geq a_{n-1}\geq |a_n|$ for type $\tipo D_n$.

By the Highest Weight Theorem, irreducible representations of $\mathfrak g$ are in correspondence with integral dominant weights.
We denote by $P^{{+}{+}}(\mathfrak g)$ the set of integral dominant weights and by $\pi_\lambda$ the irreducible representation of $\mathfrak g$ with highest weight $\lambda\in P^{{+}{+}}(\mathfrak g)$.

The first two fundamental weights are $\omega_1=\varepsilon_1$ and $\omega_2=\varepsilon_1+\varepsilon_2$. 
Hence, any non-negative integer combination of them is of the form $k\varepsilon_1+l\varepsilon_2$ for some integers $k\geq l\geq 0$.

The following notation is essential to read the weight multiplicity formulas in the next section.
We will write $\Z^n$ for the set of elements $\sum_i a_i\varepsilon_i \in P(\mathfrak g)$ such that $a_i\in\Z$ for all $i$.
For a weight $\mu=\sum_{i=1}^n a_i\varepsilon_i\in\Z^n $ and $t$ a non-negative integer, define
\begin{align}\label{eq:def_norma-contador}
 \norma{\mu}& = \sum_{i=1}^n |a_i|, &
 \contador_t(\mu) &= \#\{ i: 1\leq i\leq n,\, |a_i|=t\}.
\end{align}
Given $N\geq0$, let $\mathcal{Q}_n(N)$ be the set of all partitions of $N$ with length $\leq n$, that is
\begin{align}\label{eq:def-A(N)}
 \mathcal{Q}_n(N)&= \{ \mathbf{q}=(q_1,q_2,\dots,q_n)\in \Z^n: q_1\geq q_2\geq\dots\geq q_n\geq0,\, \textstyle \sum_{i=1}^n q_i=N \}.
\end{align}
Furthermore, for $\mathbf{q}\in\mathcal{Q}_n(N)$ and $1\leq j\leq N$ we set
\begin{align}\label{eq:def-s_j}
 s_j^{\mathbf q}&:= \#\{i:1\leq i\leq {n},\, q_i=j\},
 \\
 \label{eq:def-C^q}
 \mathcal{B}^{\mathbf q}
 &:= \{ (\beta^1_1,\beta^2_1,\beta^2_2,\beta^3_1,\beta^3_2, \beta^3_3,\dots,\beta^N_N): \, \beta^j_t\geq 0, \, \textstyle \sum_{t=1}^{j} \beta^j_t\leq s_j^{\mathbf q} \},\\
 \label{eq:def-C^q_b}
 \mathcal{A}^{\mathbf q}_\beta
 &:= \{ (\alpha^1_1,\alpha^2_1,\alpha^2_2,\alpha^3_1,\alpha^3_2, \alpha^3_3, \dots,\alpha^N_N): \, 0\leq \alpha^j_t\leq  \beta^j_t \}
 	\qquad  \text{for any } \beta\in \mathcal{B}^{\mathbf q}.
\end{align}

Throughout the article we use the convention $\binom{b}{a}=0$ if $a<0$ or $b<a$.

\section{Types $\tipo B_n$, $\tipo C_n$ and $\tipo D_n$}\label{sec:typeBCD}
In this section we consider $\mathfrak g$ a classical complex Lie algebra of types $\tipo B_n$, $\tipo C_n$ and $\tipo D_n$.
We assume that $n\geq2$ for types $\tipo B_n$ and $\tipo C_n$ and $n\geq3$ for type $\tipo D_n$.

\subsection{Main results}
We now state the three theorems which establish the weight multiplicity formulas for types $\tipo B_n$, $\tipo C_n$ and $\tipo D_n$ respectively.
The notation required was introduced in the previous section.
The formulas consider weights in $\Z^n$, since the multiplicity in $\pi_{k\varepsilon_{1}+l\varepsilon_{2}}$ of any weight in $P(\mathfrak g)\smallsetminus \Z^n$ vanishes (see Remark~\ref{rem:not-G-integral}).

\begin{theorem}\label{thmBn:mulip_bivar} \textup{(Type $\tipo B_n$)}
Let $\mathfrak g=\so(2n+1,\C)$ for some $n\geq2$. Let $k\geq l\geq 0$ integers and $\mu\in\Z^n$.
If $r(\mu):=(k+l-\norma{\mu})/2$ is negative, then $m_{\pi_{k\varepsilon_1+l\varepsilon_{2}}}(\mu)=0$, and otherwise
\begin{align*}
m_{\pi_{k\varepsilon_1+l\varepsilon_{2}}}(\mu)
=&\quad B_n(l,r(\mu),\contador_0(\mu), \dots,\contador_{l-1}(\mu))\\
&- B_n(l-1,r(\mu),\contador_0(\mu), \dots,\contador_{l-1}(\mu))\\
&- B_n(l-1,r(\mu)-1,\contador_0(\mu), \dots,\contador_{l-1}(\mu))\\
&+ B_n(l-2,r(\mu)-1,\contador_0(\mu), \dots,\contador_{l-1}(\mu)),
\end{align*}
where
 \begin{align*}
  B_n(l,r,\contador_0,\dots,\contador_{l-1})=
  & \sum_{0\leq N\leq l}\; \sum_{\mathbf{q}\in\mathcal{Q}_n(N)}\;  \sum_{\beta\in \mathcal{B}^\mathbf{q}} \; \sum_{\alpha\in \mathcal{A}^\mathbf{q}_{\beta}}
  \binom{\lfloor(l-N)/2\rfloor+n-1}{n-1} \\
  &\binom{\lfloor r- (l+N)/{2}\rfloor +\sum_{j=1}^N\sum_{i=1}^j (j+1-i)\alpha_i^j+n-1}{n-1} \\
  & \prod_{j=1}^{N} \left( 2^{s_j^{\mathbf{q}}-\sum_{i=1}^j \beta^j_i} \binom{\beta_1^j}{\alpha_1^j} \binom{n-\sum_{t=0}^{j-1}\contador_t -\sum_{r=j+1}^N\sum_{s=1}^{r-j+1} \beta_s^r}{\beta_1^j}   \right.\\
  &
  \left. \binom{\contador_0-\sum_{h=j+1}^{N}(s^{\mathbf{q}}_h-\sum_{s=1}^{h} \beta_s^h)}{s^{\mathbf{q}}_j- \sum_{t=1}^{j}\beta_t^j}  \prod_{i=2}^j \binom{\contador_{j-i+1} -\sum_{t=1}^{N-j}\beta_{i+1}^{j+t}}{\beta_i^j} \binom{\beta_i^j}{\alpha_i^j}\right).
 \end{align*}
\end{theorem}

\begin{theorem}\label{thmCn:mulip_bivar} \textup{(Type $\tipo C_n$)}
 Let $\mathfrak g=\spp(n,\C)$ for some $n\geq2$. Let $k\geq l\geq 0$ integers and $\mu\in\Z^n$.
 If $r(\mu):=(k+l-\norma{\mu})/2$ is not a non-negative integer, then $m_{\pi_{k\varepsilon_1+l\varepsilon_{2}}}(\mu)=0$, and otherwise
\begin{align*}
m_{\pi_{k\varepsilon_1+l\varepsilon_{2}}}(\mu)
=& \quad
C_n(l,r(\mu),\contador_0(\mu), \dots,\contador_{l-1}(\mu)) \\
&- C_n(l-1,r(\mu),\contador_0(\mu), \dots,\contador_{l-1}(\mu))\\
&- C_n(l-1,r(\mu)-1,\contador_0(\mu), \dots,\contador_{l-1}(\mu))\\
&+ C_n(l-2,r(\mu)-1,\contador_0(\mu), \dots,\contador_{l-1}(\mu)),
\end{align*}
 where
 \begin{align*}
  C_n(l,r,\contador_0,\dots,\contador_{l-1})= & \sum_{0\leq N\leq l,\atop N \equiv l(\operatorname{mod} 2)}\; \sum_{\mathbf{q}\in\mathcal{Q}_n(N)}\;  \sum_{\beta\in \mathcal{B}^\mathbf{q}} \; \sum_{\alpha\in \mathcal{A}^\mathbf{q}_{\beta}}
  \binom{(l-N)/2+n-1}{n-1} \\
  &\binom{r-(l+N)/2 +\sum_{j=1}^N\sum_{i=1}^j (j+1-i)\alpha_i^j+n-1}{n-1} \\
  &
  \prod_{j=1}^{N} \left( 2^{s_j^{\mathbf{q}}-\sum_{i=1}^j \beta^j_i} \binom{\beta_1^j}{\alpha_1^j} \binom{n-\sum_{t=0}^{j-1}\contador_t -\sum_{r=j+1}^N\sum_{s=1}^{r-j+1} \beta_s^r}{\beta_1^j}   \right.\\
  &
  \left. \binom{\contador_0-\sum_{h=j+1}^{N}(s^{\mathbf{q}}_h-\sum_{s=1}^{h} \beta_s^h)}{s^{\mathbf{q}}_j- \sum_{t=1}^{j}\beta_t^j}  \prod_{i=2}^j \binom{\contador_{j-i+1} -\sum_{t=1}^{N-j}\beta_{i+1}^{j+t}}{\beta_i^j} \binom{\beta_i^j}{\alpha_i^j}\right).
 \end{align*}
\end{theorem}

\begin{theorem}\label{thmDn:mulip_bivar} \textup{(Type $\tipo D_n$)}
 Let $\mathfrak g=\so(2n,\C)$ for some $n\geq3$. Let $k\geq l\geq 0$ integers and $\mu\in\Z^n$.
If $r(\mu):=(k+l-\norma{\mu})/2$ is not a non-negative integer, then $m_{\pi_{k\varepsilon_1+l\varepsilon_{2}}}(\mu)=0$, and otherwise
 \begin{align*}
m_{\pi_{k\varepsilon_1+l\varepsilon_{2}}}(\mu)
=&\quad
   D_n(l,r(\mu),\contador_0(\mu), \dots,\contador_{l-1}(\mu))\\
&- D_n(l-1,r(\mu),\contador_0(\mu), \dots,\contador_{l-1}(\mu))\\
&- D_n(l-1,r(\mu)-1,\contador_0(\mu), \dots,\contador_{l-1}(\mu))\\
&+ D_n(l-2,r(\mu)-1,\contador_0(\mu), \dots,\contador_{l-1}(\mu)),
 \end{align*}
where
 \begin{align*}
  D_n(l,r,\contador_0,\dots,\contador_{l-1})= & \sum_{0\leq N\leq l,\atop N \equiv l(\operatorname{mod} 2)}\; \sum_{\mathbf{q}\in\mathcal{Q}_n(N)}\;  \sum_{\beta\in \mathcal{B}^\mathbf{q}} \; \sum_{\alpha\in \mathcal{A}^\mathbf{q}_{\beta}}
  \binom{(l-N)/2+n-2}{n-2} \\
  &\binom{r-(l+N)/2 +\sum_{j=1}^N\sum_{i=1}^j (j+1-i)\alpha_i^j+n-2}{n-2} \\
  & \prod_{j=1}^{N} \left( 2^{s_j^{\mathbf{q}}-\sum_{i=1}^j \beta^j_i} \binom{\beta_1^j}{\alpha_1^j} \binom{n-\sum_{t=0}^{j-1}\contador_t -\sum_{r=j+1}^N\sum_{s=1}^{r-j+1} \beta_s^r}{\beta_1^j}  \right. \\
  & \left. \binom{\contador_0-\sum_{h=j+1}^{N}(s^{\mathbf{q}}_h-\sum_{s=1}^{h}\beta_s^h)}{s^{\mathbf{q}}_j-\sum_{t=1}^{j}\beta_t^j}  \prod_{i=2}^j \binom{\contador_{j-i+1} -\sum_{t=1}^{N-j}\beta_{i+1}^{j+t}}{\beta_i^j} \binom{\beta_i^j}{\alpha_i^j}\right).
 \end{align*}
\end{theorem}

\begin{remark}\label{rem:not-G-integral}
For all types considered, it turns out that $m_{\pi_{k\varepsilon_1+l\varepsilon_{2}}}(\mu) =0$ for all $\mu\in P(\mathfrak g)\smallsetminus\Z^n$.
 Indeed, the subset $\Z^n$ coincides with the set of $G$-integral weights, where $G$ is the only compact linear group $G$ whose Lie algebra is a compact real form in $\mathfrak g$.
 We have that $G$ is isomorphic to $\SO(2n+1)$, $\Sp(n)$ and $\SO(2n)$ for types $\tipo B_n$, $\tipo C_n$ and $\tipo D_n$ respectively.
 Since $k\varepsilon_1+l\varepsilon_2\in\Z^n$, the representation $\pi_{k\varepsilon_1+l\varepsilon_{2}}$ descends to a representation of $G$, and consequently their weights are in $\Z^n$ (see \cite[Lem.~5.106]{Knapp-book-beyond}).
\end{remark}

\subsection{Proofs}
This subsection contains a unified proof of Theorems~\ref{thmBn:mulip_bivar}--\ref{thmDn:mulip_bivar}.
We first need two lemmas.
The first one gives well-known closed explicit formulas for the weight multiplicities of representations having highest weights of the form $k\varepsilon_{1}$ for $k$ a non-negative integer.
A proof can be found in \cite[Lemmas~3.2, 4.3, 5.3]{LR-fundstring}.
The second lemma is the first step to prove the theorems.

\begin{lemma}\label{lem:extremecases}
Let $\mathfrak g$ be a complex Lie algebra of type $\tipo B_n$, $\tipo C_n$ or $\tipo D_n$ for some $n\geq2$. Let $k\geq0$ integer and $\mu\in \Z^n$. Then
 \begin{align}
   m_{\pi_{k\varepsilon_1}}(\mu) &=
   \tbinom{\lfloor r(\mu)\rfloor +n-1}{n-1}  \quad \text{ where } r(\mu)= \tfrac{k-\norma{\mu}}{2},
&&\text{for $\mathfrak g$ of type $\tipo B_n$,}
   \label{eqBn:multip(k)} \\
  m_{\pi_{k\varepsilon_1}}(\mu) &=
  \begin{cases}
   \binom{r(\mu)+n-1}{n-1} & \text{ if }\, r(\mu):=\frac{k-\norma{\mu}}{2}\in \N_0,\\
   0 & \text{ otherwise,}
  \end{cases}
&&\text{for $\mathfrak g$ of type $\tipo C_n$,}
 \label{eqCn:multip(k)} \\
 m_{\pi_{k\varepsilon_1}}(\mu) &=
 \begin{cases}
  \binom{r(\mu)+n-2}{n-2} & \text{ if }\, r(\mu):=\frac{k-\norma{\mu}}{2} \in\N_0,\\
  0 & \text{ otherwise,}
 \end{cases}
&&\text{for $\mathfrak g$ of type $\tipo D_n$.}
 \label{eqDn:multip(k)}
\end{align}
\end{lemma}

\begin{lemma}\label{lem:virtualring(tau)}
 Let $\mathfrak g$ be a classical Lie algebra of type $\tipo B_n$, $\tipo C_n$ or $\tipo D_n$. For integers $k\geq l\geq0$ write $\tau_{k,l}=\pi_{k\varepsilon_{1}}\otimes \pi_{l\varepsilon_{1}}$.
 Then, in the virtual ring of representations of $\mathfrak g$, we have that
 $$
 \pi_{k\varepsilon_{1}+l\varepsilon_{2}}\simeq
 \tau_{k,l}- \tau_{k+1,l-1}- \tau_{k-1,l-1}+ \tau_{k,l-2}.
 $$
\end{lemma}

\begin{proof}
 We have the fusion rule (see for instance \cite[page 510, example (2)]{KoikeTerada87})
$$
 \tau_{k,l}=
 \pi_{k\varepsilon_{1}}\otimes \pi_{l\varepsilon_{1}}\simeq \bigoplus_{j=0}^{l}\bigoplus_{i=0}^{j} \pi_{(k+l-j-i)\varepsilon_1+(j-i)\varepsilon_2}.
$$
 As an immediate consequence we obtain
 \begin{align*}
  \tau_{k,l}- \tau_{k+1,l-1}=&
  \sum_{j=0}^{l}\sum_{i=0}^{j} \pi_{(k+l-j-i)\varepsilon_1+(j-i)\varepsilon_2} - \sum_{j=0}^{l-1}\sum_{i=0}^{j} \pi_{(k+l-j-i)\varepsilon_1+(j-i)\varepsilon_2}
  = \sum_{i=0}^{l} \pi_{(k-i)\varepsilon_1+(l-i)\varepsilon_2},
\\
  \tau_{k-1,l-1}- \tau_{k,l-2} =&
  \sum_{j=0}^{l-1}\sum_{i=0}^{j}
  \pi_{(k+l-j-i-2)\varepsilon_1+(j-i)\varepsilon_2}-
  \sum_{j=0}^{l-2}\sum_{i=0}^{j} \pi_{(k+l-j-i-2)\varepsilon_1+(j-i)\varepsilon_2}\\
  =& \sum_{i=0}^{l-1} \pi_{(k+1-i)\varepsilon_1+(l-1-i)\varepsilon_2}
  = \sum_{i=1}^{l} \pi_{(k-i)\varepsilon_1+(l-i)\varepsilon_2}.
 \end{align*}
Subtracting the previous identities we obtain the desired formula.
\end{proof}

\begin{proof}[Proofs of Theorems~\ref{thmBn:mulip_bivar}--\ref{thmDn:mulip_bivar}]
Without loss of generality we can assume that $\mu \in \Z^n$ is dominant  since the Weyl group preserves weight multiplicities.
Recall that $\tau_{k,l}= \pi_{k \varepsilon_{1}} \otimes \pi_{l \varepsilon_{1}}$ for $k\geq l\geq 0$ integers.
Since
\begin{equation}\label{eq:mult-pi(k,l)=virtualring}
m_{\pi_{k\varepsilon_{1}+l\varepsilon_{2}}}(\mu) =
m_{\tau_{k,l}}(\mu)- m_{\tau_{k+1,l-1}}(\mu)- m_{\tau_{k-1,l-1}}(\mu)+ m_{\tau_{k,l-2}}(\mu)
\end{equation}
by Lemma~\ref{lem:virtualring(tau)}, we are left with the task of showing that
\begin{equation}\label{eq:mult-tau=BnCnDn}
m_{\tau_{k,l}}(\mu) =
\begin{cases}
B_n(l,r(\mu),\contador_0(\mu),\dots,\contador_{l-1}(\mu))
 &\text{ for $\mathfrak g$ of type $\tipo B_n$,}\\
C_n(l,r(\mu),\contador_0(\mu),\dots,\contador_{l-1}(\mu))
 &\text{ for $\mathfrak g$ of type $\tipo C_n$,}\\
D_n(l,r(\mu),\contador_0(\mu),\dots,\contador_{l-1}(\mu))
 &\text{ for $\mathfrak g$ of type $\tipo D_n$.}
\end{cases}
\end{equation}

It is well known that (see  \cite[Excercise~V.14]{Knapp-book-beyond})
\begin{align}\label{eq:multiptensor2}
 m_{\tau_{k,l}}(\mu)=\sum_{\eta} m_{\pi_{k\varepsilon_{1}}}(\mu-\eta)\, m_{\pi_{l\varepsilon_{1}}}(\eta),
\end{align}
where the sum is restricted to $\mathcal{P}(\pi_{l\varepsilon_{1}})$, the set of weights of $\pi_{l\varepsilon_{1}}$.
By Lemma~\ref{lem:extremecases}, the weights of $\pi_{l\varepsilon_{1}}$ are those $\eta$ such that $l-\norma{\eta}\in \N_0$ for type $\tipo B_n$ and $l-\norma{\eta}\in2\N_0$ for types $\tipo C_n$ and $\tipo D_n$.
In order to calculate $\norma{\mu-\eta}$, to determine $m_{k\varepsilon_{1}}(\mu-\eta)$,
we make a convenient partition of $\mathcal{P}(\pi_{l\varepsilon_{1}})$.

We write $\contador_t=\contador_t(\mu)$ for all $0\leq t\leq l-1$.
For $0\leq N\leq l$ integer, $\mathbf q=(q_1,\dots,q_n) \in \mathcal{Q}_n(N)$, $\beta=(\beta^1_1,\beta^2_1,\dots,\beta^N_N)\in\mathcal B^{\mathbf q}$ and $\alpha=(\alpha^1_1,\alpha^2_1,\dots,\alpha^N_N)\in\mathcal A^{\mathbf q}_\beta$ (see Section~\ref{sec:notation}
for notation), we set
\begin{align}\label{eq:setP}
\mathcal{P}^{\mathbf q}_{\beta,\alpha}= \left\{\sum_{j=1}^N \sum_{k=0}^{s_j^{\mathbf q}} b_{j,k}\varepsilon_{i_k^j}:
\begin{array}{l}
 i^j_k\neq i^{j'}_{s'} \mbox{ if } (j,k)\neq(j',k'); \, b_{j,k}=\pm j\;\text{ for all $k$};\\
 i^j_1<\dots<i^j_{\beta^j_1}\leq n-\sum\limits_{t=0}^{j-1} \contador_t < i^j_{\beta^j_1+1}<\dots <i^j_{\beta^j_1+\beta^j_2}  \\
 \leq  n-\sum\limits_{t=0}^{j-2} \contador_t
 < \dots \leq n-\contador_0 < i^j_{\sum_{i=1}^j \beta^j_i+1}  <\dots < i^j_{s^{\mathbf q}_j, \,}
 \\
 \text{for all } 1\leq j\leq N;\\
 \#\{k: \sum\limits_{t=1}^{h-1} \beta^j_{t}+1\leq k\leq \sum\limits_{t=1}^{h} \beta^j_{t}, \, b_{j,k}=j\}= \alpha^j_h
 \end{array}
 \right\}.
\end{align}
We now list some properties shared by all the elements in $\mathcal{P}^{\mathbf q}_{\beta,\alpha}$.
Let $\eta =\sum_{i=1}^n c_i\varepsilon_i\in \mathcal{P}^{\mathbf q}_{\beta,\alpha}$.
The multiset (i.e.\ a set where an element can be repeated) given by the elements $|c_1|,\dots,|c_n|$ coincides with the multiset of elements $q_1,\dots,q_n$, thus $\norma{\eta}=N$.
For a fixed $1\leq j\leq N$, the number of entries equal to $\pm j$ is $s^{\mathbf q}_j$ located as follows:
we divide the integral interval $[1,n]$ in $(j+1)$-blocks as the identity $n=(n-\sum_{t=0}^{j-1}\contador_t )+ \contador_{j-1}+\contador_{j-2}+\dots+\contador_1+\contador_0$ suggests, that is, the first block has the first $(n-\sum_{r=0}^{j-1}\contador_r )$ integers, the second block has the next $\contador_{j-1}$ elements, the third one has the next $\contador_{j-2}$ elements, and so on.
For each $1\leq t\leq j$, there are $\beta_t^j$ entries in the $t$th block equal to $\pm j$, $\alpha_t^j$ of them are positive.
In the last block there are $s^{\mathbf q}_j-\sum_{t=1}^{j-1} \beta_t^j$ entries equal to $\pm j$.

As a consequence of the previous paragraph, we have partitioned the set of weights of $\pi_{l \varepsilon_{1}}$ as
\begin{equation}\label{eq:partition}
\mathcal{P}(\pi_{l\varepsilon_{1}})= \bigcup_N \, \bigcup_{\mathbf q\in\mathcal{Q}_n(N)}\; \bigcup_{\beta\in\mathcal B^{\mathbf q}}\; \bigcup_{\alpha\in\mathcal A^{\mathbf q}_\beta} \; \mathcal{P}^{\mathbf q}_{\beta,\alpha}
\end{equation}
where the first union is over $N\in\N_0$ satisfying $l-N\in \N_0$ for type $\tipo B_n$ and $l-N\in 2\N_0$ for types $\tipo C_n$ and $\tipo D_n$.
All the unions are disjoint.

Fix an integer $0\leq N\leq l$, $\mathbf q\in\mathcal{Q}_n(N)$, $\beta\in\mathcal B^{\textbf{q}}$, $\alpha\in\mathcal A^{\textbf{q}}_\beta$ and $\eta\in\mathcal{P}^{\mathbf q}_{\beta,\alpha}$.
One may check that
\begin{align*}
  \norma{\mu-\eta} &=  k+l-2r + \sum_{j=1}^N \sum_{i=1}^j \big( j(\beta^j_i-\alpha^j_i)+(2i-j-2)\alpha^j_i\big) + \sum_{j=1}^N j(s^{\mathbf q}_j-\sum_{i=1}^j \beta^j_i)\\
  &= k+l-2r+ \sum_{j=1}^N \sum_{i=1}^j 2(i-j-1)\alpha^j_i + js^{\mathbf q}_j \\
  &= k+l +N- 2\big(r+ \sum_{j=1}^N \sum_{i=1}^j (j+1-i)\alpha^j_i   \big)\\
  &= k-2\big( r+ \sum_{j=1}^N \sum_{i=1}^j (j+1-i)\alpha^j_i - (l+N)/2\big).
\end{align*}
Since $m_{\pi_{l\varepsilon_{1}}}(\eta)$ and $m_{\pi_{k\varepsilon_{1}}}(\mu-\eta)$ are given in Lemma~\ref{lem:extremecases} in terms of $l-\norma{\eta}$ and $k-\norma{\mu-\eta}$ respectively, it follows that $m_{\pi_{l\varepsilon_{1}}}(\eta)$ and $m_{\pi_{k\varepsilon_{1}}}(\mu-\eta)$ are constant, independent of the choice of $\eta\in \mathcal{P}^{\mathbf q}_{\beta,\alpha}$.

From the above fact, the partition \eqref{eq:partition} and the formula \eqref{eq:multiptensor2} we conclude that
\begin{align}\label{eq:multiptensor3}
 m_{\tau_{k,l}}(\mu) = & \sum_{N}\; \sum_{\mathbf{q}\in\mathcal{Q}_n(N)}\;  \sum_{\beta\in \mathcal{B}^\mathbf{q}} \; \sum_{\alpha\in \mathcal{A}^\mathbf{q}_{\beta}}
 m_{\pi_{k\varepsilon_{1}}} (\mu-\eta_{\beta,\alpha}^{\mathbf{q}}) \, m_{\pi_{l\varepsilon_{1}}} (\eta_{\beta,\alpha}^{\mathbf{q}})\, \# \mathcal{P}^{\mathbf q}_{\beta,\alpha},
\end{align}
where $\eta_{\beta,\alpha}^{\mathbf{q}}$ is any element in $\mathcal{P}^{\mathbf q}_{\beta,\alpha}$, and the first sum is over $N\in\N_0$ satisfying $l-N\in \N_0$ for type $\tipo B_n$ and $l-N\in 2\N_0$ for types $\tipo C_n$ and $\tipo D_n$.

By tedious but straightforward combinatorial arguments we have
\begin{align*}
\#\mathcal{P}^{\mathbf q}_{\beta,\alpha}= &
2^{\sum_{j=1}^N(s_j^{\mathbf q}-\sum_{i=1}^j \beta^j_i)}  \binom{n-\sum_{j=0}^{N-1}\contador_j}{\beta^N_1}\binom{\contador_{N-1}}{\beta^N_2}\cdots \binom{\contador_1}{\beta^N_N}\binom{\contador_0}{s^{\mathbf q}_N-\sum_{j=1}^N\beta_j^N}\\
& \binom{n-\sum_{j=0}^{N-2}\contador_j-\beta_1^N-\beta_2^N}{\beta^{N-1}_1}\binom{\contador_{N-2}-\beta^N_3}{\beta^{N-1}_2}\cdots \binom{\contador_1-\beta^N_N}{\beta^{N-1}_{N-1}}\binom{\contador_0-(s^{\mathbf q}_N-\sum_{j=1}^N\beta_j^N)}{s^{\mathbf q}_{N-1}-\sum_{j=1}^{N-1}\beta_j^{N-1}} \\
& \cdots \binom{n-\contador_0-\sum_{j=2}^N\sum_{i=1}^j\beta_i^j}{\beta^1_1} \binom{\contador_0-\sum_{j=2}^N(s^{\mathbf q}_j-\sum_{i=1}^j \beta_i^j)}{s^{\mathbf q}_{1}-\beta_1^1} \binom{\beta^1_1}{\alpha^1_1} \binom{\beta^2_1}{\alpha^2_1}\cdots \binom{\beta^N_N}{\alpha^N_N}\\
= &
\prod_{j=1}^{N} \left(2^{s_j^{\mathbf q}-\sum_{i=1}^j \beta^j_i} \binom{n-\sum_{t=0}^{j-1}\contador_t -\sum_{r=j+1}^N\sum_{s=1}^{r-j+1} \beta_s^r}{\beta_1^j}
\right.\\
&\quad \quad\left. \binom{\contador_0-\sum_{r=j+1}^{N}(s^{\mathbf q}_r-\sum_{s=1}^{r}\beta_s^r)}{s^{\mathbf q}_j-\sum_{t=1}^{j}\beta_t^j}
\binom{\beta_1^j}{\alpha_1^j}  \prod_{i=2}^j \binom{\contador_{j-i+1} -\sum_{t=1}^{N-j}\beta_{i+1}^{j+t}}{\beta_i^j} \binom{\beta_i^j}{\alpha_i^j} \right).
\end{align*}

Replacing in \eqref{eq:multiptensor3} the values of $m_{\pi_{k\varepsilon_{1}}}(\mu-\eta_{\beta,\alpha}^{\mathbf{q}})$ and $m_{\pi_{l\varepsilon_{1}}} (\eta_{\beta,\alpha}^{\mathbf{q}})$ given by Lemma~\ref{lem:extremecases} and $\#\mathcal{P}^{\mathbf q}_{\beta,\alpha}$ by the above expression, we obtain the desired weight multiplicity formula for $\tau_{k,l}$.
According to \eqref{eq:mult-pi(k,l)=virtualring}, the proofs of Theorems~\ref{thmBn:mulip_bivar}, \ref{thmCn:mulip_bivar} and \ref{thmDn:mulip_bivar} are complete.
\end{proof}

\subsection{Computational comparison}\label{subsec:comparison}
We now include a non-serious computational comparison between the weight multiplicity formulas in Theorems~\ref{thmBn:mulip_bivar}--\ref{thmDn:mulip_bivar} and Freudenthal's formula (see for instance \cite[\S22.3]{Humphreys-book}).
We use the open-source mathematical software \texttt{Sage}~\cite{Sage} and its algebraic combinatorics features developed by the \texttt{Sage-Combinat} community~\cite{Sage-Combinat}, which has implemented Freudenthal's formulas.
The source code containing the bivariate algorithm can be found in the public project \cite{Publicbivariate} available in \texttt{CoCalc}. (To see the corresponding hyperlink go to the electronic version of this article.)

The word `non-serious' in the previous paragraph has been added for several reasons that we now explain.
The formulas proved above have been implemented in \texttt{Sage} by the first named author, who lacks computer programming skills.
Thus, their implementations are done poorly and inefficiently.
On the contrary, the \texttt{Sage-Combinat} community programmed Freudenthal's formula in \texttt{Sage} in a very efficient way.
Furthermore, the calculations have been made using an old version of \texttt{Sage}~\cite{Sage} and a slow computer.
	
The implementation of Freudenthal's formula in \texttt{Sage}, called \emph{Freudenthal algorithm} in the sequel, returns all the weights with their corresponding multiplicities.
We suspect that this tactic is due to a matter of efficiency since Freudenthal's formula is defined recursively.
On the other hand, Theorem~\ref{thmBn:mulip_bivar}--\ref{thmDn:mulip_bivar} compute the multiplicity of a single weight.
Thus, in order to make a fair comparison between them, the bivariate algorithm will also determine the set of weights of $\pi_{k \varepsilon_{1}+l\varepsilon_2}$.
To this end, we first find a subset of $\Z^n$ containing the set of weights of $\pi_{k \varepsilon_{1}+l\varepsilon_2}$, namely, $\{\mu\in\Z^n: \norma{\mu}\leq k+l\}$.
Here is a summary of the algorithm.

\begin{algorithm}[{Bivariate algorithm}]\label{algorithm}\ \\
		\textsc{Input:} $\mathfrak g$ a classical complex Lie algebra of type $\tipo B_n$ or $\tipo C_n$ with $n\geq2$, or $\tipo D_n$ with $n\geq3$, and $k\geq l$ non-negative integers. \\
		\textsc{Output}: the sequence of pairs $[\mu, m_{\pi_{k\varepsilon_{1}+l\varepsilon_{2}}}(\mu)]$, where $\mu$ runs over every weight of the representation $\pi_{k\varepsilon_1+l\varepsilon_2}$ of $\mathfrak g$ and $m_{\pi_{k\varepsilon_{1}+l\varepsilon_{2}}}(\mu)$ is its  multiplicity.
		\begin{enumerate}
			\item Initialize $S$ as an empty list.
			\item Determine the set $P$ of vectors $\mu=(a_1,\dots,a_n) \in\Z^n$ such that $\norma{\mu}\leq k+l$ and $a_1\geq a_2\geq \dots \geq a_n\geq 0$.
			\item Run over all elements $\mu$ in $P$.
			\item Compute $m_{\pi_{k\varepsilon_{1}+l\varepsilon_{2}}}(\mu)$ by Theorems~\ref{thmBn:mulip_bivar}--\ref{thmDn:mulip_bivar}.
			\item In case $m_{\pi_{k\varepsilon_{1}+l\varepsilon_{2}}}(\mu)>0$, determine the orbit of $\mu$ by the group $W_n\simeq \op{Sym}(n)\times \{\pm1\}^n$, which acts by permutations and multiplication by $\pm1$ on its entries.
			\item For each $\nu$ in the above orbit, add in $S$ the entry $[\nu, m_{\pi_{k\varepsilon_{1}+l\varepsilon_{2}}}(\mu)]$.
			\item Return $S$.
		\end{enumerate}
\end{algorithm}

\begin{remark}\label{rem:dominantes_Dn}
Notice that the set of dominant weights for $\pi_{k \varepsilon_{1}+l\varepsilon_2}$ is included in $P$ introduced in (ii) when $\mathfrak g$ is of type $\tipo B_n$ and $\tipo C_n$.
Although this fact is not true for $\mathfrak g$ of type $\tipo D_n$, each remaining element has the form $\bar \mu:=(a_1,\dots,a_{n-1},-a_n)$ for some $\mu=(a_1,\dots,a_n)$ in $P$ with $a_n>0$, and it satisfies $m_{\pi_{k\varepsilon_{1}+l\varepsilon_{2}}}(\bar \mu)=m_{\pi_{k\varepsilon_{1}+l\varepsilon_{2}}}(\mu)$ since $n\geq3$.
Consequently, step (v) obtains all the weights of $\pi_{k \varepsilon_{1}+l\varepsilon_2}$ when $\mathfrak g$ is of type $\tipo B_n$ and $\tipo C_n$ for $n\geq2$ and $\tipo D_n$ for $n\geq3$.
Likewise, the group $W_n$ introduced in (v) coincides with the Weyl group when $\mathfrak g$ is of type $\tipo B_n$ or $\tipo C_n$.
For $\mathfrak g$ of type $\tipo D_n$ and $n\geq3$, the Weyl group is isomorphic to $\op{Sym}(n)\times \{\pm1\}^{n-1}$, thus it is strictly included in $W_n$.
This fact is consistent with the previous comment on the set of dominant weights that is not contained in $P$.
\end{remark}

Table~\ref{tabla:comparison} displays the times (in seconds) required by both, the bivariate and Freudenthal, algorithms for different choices of $n$, $k$ and $l$.
Let us introduce the notation $B(\tipo X_n,k,l)$ for the time required by our implementation in \texttt{Sage} of the bivariate algorithm for $\mathfrak g$ of type $\tipo X_n$ ($=\tipo B_n$, $\tipo C_n$ or $\tipo D_n$) and the irreducible representation of $\mathfrak g$ having highest weight $k\varepsilon_1+l\varepsilon_{2}$.
Similarly, write $F(\tipo X_n,k,l)$ for the corresponding required time for the implementation in \texttt{Sage} of Freudenthal algorithm.
This abuse of notation (the numbers are periods of time not uniquely determined) will be advantageous to express the numerical conclusions.

We now indicate some conclusions evidenced by the numerical experiments.
It is clear that $B(\tipo X_n,k,l)$ is much smaller than $F(\tipo X_n,k,l)$ for coherent small values of $n$, $k$ and $l$.
Furthermore, the function $n\mapsto B(\tipo X_n,k,l)/F(\tipo X_n,k,l)$ seems to be increasing for any fixed choice of $\tipo X$, $k$ and $l$.
Moreover, for $n$ big enough, one would have $B(\tipo X_n,k,l)<F(\tipo X_n,k,l)$.

On one hand, we see that $F(\tipo D_n,k,l) < F(\tipo C_n,k,l) < F(\tipo B_n,k,l)$ and the gaps among them increase when $n$ grows.
The reason is that Freudenthal's formula depends heavily on the root system associated to $\mathfrak g$, which is simpler for type $\tipo D_n$ and more complicated for type $\tipo B_n$.
On the other hand, $B(\tipo C_n,k,l)$ and $B(\tipo D_n,k,l)$ look similar and $B(\tipo B_n,k,l)$ larger.
In this case, the reason is the number of weights.
Roughly speaking, the set of weights of $\pi_{k\varepsilon_{1}+l\varepsilon_{2}}$ is almost equal to $\{\mu\in\Z^n : \norma{\mu}\leq k+l, \norma{\mu}\equiv k+l\pmod 2\}$ for types $\tipo C_n$ and $\tipo D_n$ and to $\{\mu\in\Z^n : \norma{\mu}\leq k+l\}$ for types $\tipo B_n$.
In fact, this is a consequence of $\norma{\alpha}=2$ for every root $\alpha$ in types $\tipo C_n$ and $\tipo D_n$ and $\norma{\alpha}\in\{1,2\}$ for every root $\alpha$ in type $\tipo B_n$.
Summing up, the bivariate algorithm is not sensible to the number of roots in the corresponding root system, but it is sensible to the one-norm of the roots.

Throughout this paragraph fix a type $\tipo X_n$.
The times required by both algorithms depend on $k+l$.
In fact, the set of weights of $\pi_{k\varepsilon_{1}+l\varepsilon_{2}}$ does not vary considerably among the different choices of $k$ and $l$ with $k+l$ fixed.
Likewise, Freudenthal's formula is slightly faster when $l$ grows since the size of the set of weights decreases.
However, the bivariate algorithm strongly depends on $l$.
Indeed, as this algorithm involves partitions of all non-negative integers less than or equal to $l$, its speed reduces when $l$ increases.
In conclusion, fixing the value $m=k+l$, the function $l\mapsto B(\tipo X_n,m-l,l)/F(\tipo X_n,m-l,l)$ attains its minimum when $l$ is as large as possible, that is, when $l=k$ or $l=k-1$ according to the parity of $k+l$.
This situation is exemplified in Table~\ref{tabla:k+lfijo}.

The authors believe that the weight multiplicity formulas in Theorems~\ref{thmBn:mulip_bivar}--\ref{thmDn:mulip_bivar} could be implemented on new versions of \texttt{Sage}.
Bivariate representations are a non-trivial class of irreducible representations, which frequently appear on users' calculations.
Not only the time required by the bivariate algorithm for $n$ large enough is reduced, but there is also a great advantage in the possibility of calculating the multiplicity of a single weight in a very short period of time.
For example, when $\mathfrak g$ is of type $\tipo D_5$, $k=20$ and $l=6$, its implemented program in \texttt{Sage} takes only between 0.40 and 0.65 seconds for each single weight $\mu$.
Furthermore, the efficiency of the algorithm improves significantly when it returns only the multiplicities of dominant weights (i.e.~step (v) is omitted in Algorithm~\ref{algorithm}), which is in general what users really need.
This can be appreciated in the fourth column of Table~\ref{tabla:comparison}, denoted by $\tipo D_n(\star)$.
There, we list the times required by this simplified version of the bivariate algorithm 
for $\mathfrak g$ of type $\tipo D_n$.
Of course, the fact that the bivariate algorithm works only for particular simple complex Lie algebras and bivariate representations is a big disadvantage.

\begin{table}
\caption{Computational comparison between bivariate algorithm and Freudenthal's formula.} \label{tabla:comparison}
\begin{tabular}{ccc|rrrr|rrr}
&&&\multicolumn{4}{c}{time bivariate} & \multicolumn{3}{|c}{time Freudenthal}   \\
$n$&$k$&$l$& \multicolumn{1}{c}{$\tipo B_n$} & \multicolumn{1}{c}{$\tipo C_n$} & \multicolumn{1}{c}{$\tipo D_n$} & \multicolumn{1}{c}{$\tipo D_n$$(\star)$} & \multicolumn{1}{|c}{$\tipo B_n$} & \multicolumn{1}{c}{$\tipo C_n$} & \multicolumn{1}{c}{$\tipo D_n$}  \\ \hline\hline
 2& 5&3&    0.15 &    0.07 &         &         &     0.32 &     0.13 &          \\
 3& 5&3&    0.26 &    0.14 &    0.15 &    0.13 &     3.88 &     1.87 &     1.27 \\
 4& 5&3&    0.46 &    0.35 &    0.28 &    0.19 &    32.58 &    14.90 &    12.05 \\
 5& 5&3&    0.99 &    0.67 &    0.62 &    0.23 &   187.43 &    94.08 &    79.82 \\
 6& 5&3&    2.82 &    1.89 &    1.78 &    0.24 &   876.17 &   527.69 &   451.43 \\
 7& 5&3&    6.94 &    5.22 &    4.74 &    0.27 &  3436.25 &  1898.23 &  1961.54 \\
 8& 5&3&   17.77 &   14.11 &   12.59 &    0.36 &          &          &          \\
 9& 5&3&   43.23 &   35.47 &   32.11 &    0.51 &          &          &          \\
10& 5&3&   97.55 &   87.67 &   78.66 &    0.84 &          &          &          \\ \hline
 2&10&3&    0.29 &    0.13 &         &         &     1.92 &     0.48 &          \\
 3&10&3&    0.82 &    0.49 &    0.45 &    0.37 &    23.68 &    10.23 &     7.85 \\
 4&10&3&    2.09 &    1.16 &    1.22 &    0.58 &   291.63 &   130.93 &   108.61 \\
 5&10&3&    8.30 &    5.06 &    4.84 &    0.80 &  2630.09 &  1193.45 &  1028.42 \\
 6&10&3&   38.42 &   24.86 &   23.90 &    1.16 &          &          &          \\
 7&10&3&  183.73 &  146.06 &  126.82 &    1.96 &          &          &          \\ \hline
 2&50&3&    3.47 &    1.79 &         &         &    78.13 &    28.54 &          \\
 3&50&3&   36.40 &   17.62 &   17.50 &   13.50 &  5146.69 &  2108.21 &  1578.78 \\
 4&50&3&  472.14 &  325.76 &  267.25 &   58.00 &          &          &          \\ \hline
 2& 6&6&    2.20 &    1.32 &         &         &     0.50 &     0.19 &          \\
 3& 6&6&    9.28 &    5.35 &    5.23 &    5.21 &    11.90 &     4.74 &     3.53 \\
 4& 6&6&   19.29 &   11.19 &   11.44 &   11.09 &   157.23 &    67.93 &    54.79 \\
 5& 6&6&   30.81 &   18.05 &   18.04 &   15.78 &  1443.41 &   663.86 &   553.57 \\
 6& 6&6&   53.76 &   32.65 &   32.64 &   19.50 &          &          &          \\ \hline
 2&10&6&    3.58 &    2.04 &         &         &     1.66 &     0.62 &          \\
 3&10&6&   18.14 &    9.86 &    9.91 &    9.77 &    43.50 &    18.34 &    13.59 \\
 4&10&6&   44.69 &   25.20 &   24.96 &   23.98 &   695.71 &   298.50 &   243.70 \\
 5&10&6&   87.71 &   49.55 &   49.36 &   38.89 &  8114.00 &  3571.44 &  2966.84 \\
 6&10&6&  235.76 &  158.26 &  133.51 &   52.88 &          &          &          \\ \hline
 2&20&6&    8.77 &    4.61 &         &         &     8.41 &     3.16 &          \\
 3&20&6&   63.71 &   33.25 &   33.16 &   32.61 &   312.83 &   129.92 &    98.02 \\
 4&20&6&  216.46 &  115.58 &  117.99 &  109.08 &  7486.63 &  3199.29 &  2620.10 \\
 5&20&6&  654.96 &  390.98 &  393.92 &  220.11 &          &          &          \\ \hline
 2&15&9&   40.04 &   22.33 &         &         &     4.93 &     1.81 &          \\
 3&15&9&  390.59 &  209.87 &  209.12 &  208.37 &   191.24 &    78.85 &    59.11 \\
 4&15&9& 1594.63 &  865.10 &  853.83 &  851.50 &  4710.03 &  1908.34 &  1642.50 \\
 5&15&9& 3794.15 & 2112.98 & 2051.99 & 1962.57 & 71389.97 & 32013.22 & 28179.33 \\ \hline
 2&50&9&  231.18 &  116.93 &         &         &    96.16 &    35.20 &          \\
 3&50&9& 4800.55 & 2423.05 & 2553.71 & 2492.15 &  7851.47 &  3117.85 &  2346.84 \\
\end{tabular}

\medskip

Each column shows, for the corresponding algorithm and type $\tipo X$, the required time for returning the set of weights with multiplicities of the representation $\pi_{k\varepsilon_1+l\varepsilon_2}$ of $\mathfrak g$ of type $\tipo X_n$ according to the row.
The column $\tipo D_n$$(\star)$ refers to the version of the bivariate algorithm returning only the dominant weights. 
\end{table}

\begin{table}
	\caption{Comparison among representations of $\mathfrak g$ of type $\tipo D_4$ with $k+l=14$ fixed.} \label{tabla:k+lfijo}
	$
	\begin{array}{c|rrrrrrrr}
	l&\multicolumn{1}{c}{0}&\multicolumn{1}{c}{1}& \multicolumn{1}{c}{2}& \multicolumn{1}{c}{3}& \multicolumn{1}{c}{4}& \multicolumn{1}{c}{5}& \multicolumn{1}{c}{6}& \multicolumn{1}{c}{7} \\ \hline
	B(\tipo D_4,14-l,l) &  1.03&  1.04&  1.09&  1.69&  3.28&  7.45& 17.50& 39.99\\
	F(\tipo D_4,14-l,l) & 152.84& 152.77& 152.60& 151.50& 146.38& 137.41& 124.40& 106.22\\
	\end{array}
	$
\end{table}

\subsection{Closed explicit weight formulas in particular cases}\label{subsec:closedformulas}
The weight multiplicity formulas obtained in Theorems~\ref{thmBn:mulip_bivar}--\ref{thmDn:mulip_bivar} are not closed expressions because they involve a sum over partitions of non-negative integers.
However, in some particular cases (e.g.\ small values of $l$, particular choices of $\mu$) it is possible to write out the partitions, and therefore the formulas become closed expressions.
For example, if $l=0$ then the formulas reduce to the closed explicit expressions in Lemma~\ref{lem:extremecases}.

When $l=1$, only sums over the set of partitions of $0$ or $1$ are involved.
These sets have exactly one element, so the sums disappear.
For example, when $\mathfrak g$ is of type $\tipo D_n$, we get
\begin{align}
m_{\pi_{k\varepsilon_1+\varepsilon_{2}}}(\mu)
=&
D_n(1,r(\mu),\contador_0(\mu))
- D_n(0,r(\mu))
- D_n(0,r(\mu)-1)\\
= &
\sum_{\beta_1^1=0}^1 \; \sum_{\alpha_1^1=0}^{1-\beta_1^1}
\binom{r-1 + \alpha_1^1+n-2}{n-2}  2^{1-\beta^1_1} \binom{\beta_1^1}{\alpha_1^1} \binom{n-\contador_0(\mu) }{\beta_1^1}
\binom{\contador_0(\mu)}{1-\beta_1^1}  \notag\\
&- \binom{r(\mu) +n-2}{n-2} -\binom{r(\mu)-1 +n-2}{n-2} \notag
\end{align}
for every $\mu\in\Z^n$ satisfying that $r(\mu)=(k+1-\norma{\mu})/2$ is a non-negative integer.
Notice that this formula is a particular case of  \cite[Thm.~4.1]{LR-fundstring}.

Similarly, when $l=2$ there are only sums over the set of partitions of $N$ for $N=0,1,2$.
Since $2=2$ and $2=1+1$ are the only partitions of $2$, the corresponding sum splits in two.
We now state the multiplicity formula for $l=2$ and type $\tipo D_n$.
We pick type $\tipo D_n$ for citing purposes.

\begin{corollary}\label{corDn:for(l=2)}
 Let $\mathfrak g=\so(2n,\C)$ for $n\geq3$, let $k\geq2$ integer and let $\mu\in \Z^n$.
 If $r(\mu):=(k+2-\norma{\mu})/2$ is a non-negative integer, then
  \begin{align*}
   m_{\pi_{k\varepsilon_1+2\varepsilon_2}}(\mu)
   &= \binom{r(\mu)+n-4}{n-2}\left(2\contador_0(\mu)(n-1)+\binom{n-\contador_0(\mu)}{2}\right) \\
   &\quad+ \binom{r(\mu)+n-3}{n-2}\left(2\contador_0(\mu)(n-\contador_0(\mu))+ \contador_1(\mu) -n+2\binom{n-\contador_0(\mu)}{2}\right) \\
   &\quad+\binom{r(\mu)+n-2}{n-2}\left(\binom{n-\contador_0(\mu)}{2}-\contador_1(\mu)\right),
  \end{align*}
and $m_{\pi_{k\varepsilon_1+2\varepsilon_2}}(\mu)=0$ otherwise.
\end{corollary}

Furthermore, we can obtain a closed explicit multiplicity formula for the weight $\mu=0$ in the representation $\pi_{k\varepsilon_1+l\varepsilon_2}$ of $\mathfrak g$.
We next state the formulas for types $\tipo B_n$, $\tipo C_n$ and $\tipo D_n$, but we prove it only for the case $\tipo D_n$, since types $\tipo B_n$ and $\tipo C_n$ follow in a similar way.

\begin{corollary}\label{corDn:mu=0} \textup{(Type $\tipo D_n$)}
 Let $\mathfrak g=\so(2n,\C)$ for some $n\geq 3$ and let $k\geq l\geq0$ integers.
 We have that $m_{\pi_{k\varepsilon_1+l\varepsilon_2}}(0)=0$ if $k+l$ is odd.
 Moreover, if $k+l$ is even, then
 \begin{align*}
  m_{\pi_{k\varepsilon_1+l\varepsilon_2}}(0)
  &= 2\sum_{0\leq N\leq l}\; (-1)^{N+l} \, R(n,k,l,N) \, \binom{\lfloor(l-N)/2\rfloor+n-2}{n-2} \\
  &\quad \binom{\lfloor(k-N+1)/2\rfloor +n-2}{n-2}
  \sum_{t=0}^n \binom{n}{t}\binom{N-t+n-1}{n-1},
 \end{align*}
where
\begin{align*}
 R(n,k,l,N)=\begin{cases}
  \dfrac{l-N+n-2}{l-N+2n-4}
&\mbox{ if } N \equiv l \pmod 2,
\\[4mm]
  \dfrac{k+1-N+n-2}{k+1-N+2n-4}
&\mbox{ if } N \equiv l+1\pmod 2.
  \end{cases}
\end{align*}
\end{corollary}

\begin{proof}
The asserted formula can be obtained by Theorem~\ref{thmDn:mulip_bivar}.
However, we will prove it in a simplified way, by following the proof of Theorem~\ref{thmDn:mulip_bivar}.
The reason is that the partition in \eqref{eq:partition} of the set of weights of $\pi_{l\varepsilon_{1}}$ is (unnecessarily) too fine for $\mu=0$.
By \eqref{eq:mult-pi(k,l)=virtualring}, we have that
 \begin{equation}\label{eq:mult-pi(k,l)=virtualring2}
 m_{\pi_{k\varepsilon_{1}+l\varepsilon_{2}}}(0) =
 m_{\tau_{k,l}}(0)- m_{\tau_{k+1,l-1}}(0)- m_{\tau_{k-1,l-1}}(0)+ m_{\tau_{k,l-2}}(0).
 \end{equation}
As before, for arbitrary $k\geq l\geq0$ integers, it holds
 \begin{align*}
  m_{\tau_{k,l}}(0)=\sum_{\eta} m_{\pi_{k\varepsilon_{1}}}(-\eta)\, m_{\pi_{l\varepsilon_{1}}}(\eta),
 \end{align*}
 where the sum is restricted to the weights of $\pi_{l\varepsilon_{1}}$.
 From Lemma~\ref{lem:extremecases}, we see that $\eta$ is a weight of $\pi_{l\varepsilon_{1}}$ if and only if $l-\norma{\eta}\in2\N_0$.
 For such a weight $\eta$, $m_{\pi_{k\varepsilon_{1}}}(-\eta)=0$ unless $2\N_0\ni k-\norma{-\eta} = (k-l)+(l-\norma{\eta})$, equivalently $k-l\in2\N_0$.
 We conclude that $m_{\tau_{k,l}}(0)=0$ if $k+l$ is odd.
 Moreover, $m_{\pi_{k\varepsilon_{1}+l\varepsilon_{2}}}(0)=0$ if $k+l$ is odd by \eqref{eq:mult-pi(k,l)=virtualring2}.

 We now proceed to compute $m_{\tau_{k,l}}(0)$ for arbitrary $k\geq l\geq 0$ integers satisfying that $k+l$ is even.
 Fix $N\in\N_0$ such that $l-N\in2\N_0$. For each $\eta\in\Z^n$ with $\norma{\eta}=N$ we know that $m_{\pi_{k\varepsilon_{1}}}(-\eta)$ and $m_{\pi_{l\varepsilon_{1}}}(\eta)$ are constant, independent of the choice of $\eta$.
 Hence,
 \begin{align}\label{eq:multiptensor4}
m_{\tau_{k,l}}(0)=
\sum_{0\leq N\leq l, \atop N\equiv l\textrm{(mod $2$)}}
  \binom{(l-N)/2+n-2}{n-2} \binom{(k-N)/2+n-2}{n-2}\#\{\eta\in\Z^n: \norma{\eta}=N\}.
 \end{align}
It is well known (see for instance \cite[\S2.5]{BeckRobins-book}) that $\#\{\eta\in\Z^n: \norma{\eta}=N\}=\sum_{t=0}^n \binom{n}{t}\binom{N-t+n-1}{n-1}$.
Thus, by replacing \eqref{eq:multiptensor4} in \eqref{eq:mult-pi(k,l)=virtualring2} one obtains the desired formula.
\end{proof}

\begin{corollary}\label{corCn:mu=0}  \textup{(Type $\tipo C_n$)}
	Let $\mathfrak g=\spp(n,\C)$ for some $n\geq 2$ and let $k\geq l\geq0$ integers.
	We have that $m_{\pi_{k\varepsilon_1+l\varepsilon_2}}(0)=0$ if $k+l$ is odd.
	Moreover, if $k+l$ is even then
	\begin{align*}
		m_{\pi_{k\varepsilon_1+l\varepsilon_2}}(0)
		&= 2\sum_{0\leq N\leq l}\; (-1)^{N+l} \, R(n+1,k,l,N) \, \binom{\lfloor(l-N)/2\rfloor+n-1}{n-1} \\
		&\quad \binom{\lfloor(k-N+1)/2\rfloor +n-1}{n-1}
		\sum_{t=0}^n \binom{n}{t}\binom{N-t+n-1}{n-1},
	\end{align*}
where $R(n,k,l,N)$ is as in Corollary~\ref{corDn:mu=0}.
\end{corollary}

\begin{corollary}\label{corBn:mu=0}  \textup{(Type $\tipo B_n$)}
	Let $\mathfrak g=\so(2n+1,\C)$ for some $n\geq 2$ and let $k\geq l\geq0$ integers.
	Then
	\begin{align*}
		m_{\pi_{k\varepsilon_1+l\varepsilon_2}}(0)
		&= \sum_{0\leq N\leq l}\, (-1)^{N+l}S(n,k,l,N) \, \binom{\lfloor(l-N)/2\rfloor+n-1}{n-1} \\
		&\quad \binom{\lfloor(k+1-N)/2\rfloor +n-1}{n-1}
		\sum_{t=0}^n \binom{n}{t}\binom{N-t+n-1}{n-1},
	\end{align*}
	where
	\begin{align*}
		S(n,k,l,N)=\begin{cases}
			1-\dfrac{\lfloor(l-N)/2\rfloor\lfloor(k+1-N)/2\rfloor}{(\lfloor(l-N)/2\rfloor+n-1)(\lfloor(k+1-N)/2\rfloor+n-1)}
			&\mbox{ if }  k+l \mbox{ is even},
			\\[4mm]
			\dfrac{\lfloor(k+1-N)/2\rfloor}{\lfloor(k+1-N)/2\rfloor+n-1}-\dfrac{\lfloor(l-N)/2\rfloor}{\lfloor(l-N)/2\rfloor+n-1}
			&\mbox{ if } k+l \mbox{ is odd}.
				\end{cases}
		\end{align*}
\end{corollary}

\subsection{Remarks}
We end this section with a few remarks.

\begin{remark}\label{remDn:n=2}
The weight multiplicity formula for type $\tipo D_n$ in Theorem~\ref{thmDn:mulip_bivar} also holds when $n=2$ with $\pi_{k\varepsilon_{1}+l\varepsilon_{2}}$ replaced by $\pi_{k\varepsilon_{1}+l\varepsilon_{2}} \oplus \pi_{k\varepsilon_{1}-l\varepsilon_{2}}$.
It is important to note that $\so(4,\C)$ (type $\tipo D_2$) is not simple.
Indeed, $\so(4,\C)\simeq\sll(2,\C)\oplus\sll(2,\C)$ or $\tipo D_2=\tipo A_1\oplus \tipo A_1$.
Hence, a weight multiplicity formula for the representations $\pi_{k\varepsilon_1\pm l\varepsilon_{2}}$ with $k\geq l\geq 0$ of $\so(4,\C)$ can be obtained using this fact.
\end{remark}

\begin{remark}\label{rem:Maddox}
Maddox in \cite{Maddox14} determined a weight multiplicity formula for any bivariate representation for $\mathfrak g$ of type $\tipo C_n$.
Her expression (\cite[Thm.~4.3]{Maddox14}) looks more elegant than the one in Theorem~\ref{thmCn:mulip_bivar}.
However, it includes a sum over ordered partitions of $r(\mu)$ of length $n$ and another sum over the subsets of a set of $2n$ elements.
In conclusion, her shorter formula hides in the mentioned sums the involved terms appearing in the expression given in Theorem~\ref{thmCn:mulip_bivar}.
Furthermore, the neat dependence condition in Theorem~\ref{thm:depending} does not follows immediately from \cite[Thm.~4.3]{Maddox14}.

We now compare Maddox's method with ours.
Both employ the expression in Lemma~\ref{lem:virtualring(tau)} for an irreducible representation as sum of tensor products in the virtual ring of representations.
The significant difference arises in the calculation of the weight multiplicity in a tensor product.
Roughly speaking, the proofs of Theorems~\ref{thmBn:mulip_bivar}--\ref{thmDn:mulip_bivar} use the identity \eqref{eq:multiptensor2} and then a convenient partition of the set of weights of the small component in the tensor product.
On the other hand, Maddox makes use of $\tau_{k,l} = \pi_{k \varepsilon_{1}}\otimes \pi_{l \varepsilon_{1}} \simeq \op{Sym}^k(\C^{2n})\otimes \op{Sym}^l(\C^{2n})$ for $\mathfrak g$ of type $\tipo C_n$, and counts the weight vectors in terms of a function which has a combinatorial expression.
\end{remark}

\begin{remark}
There are in the literature several algorithms to compute weight multiplicities. 
The one based on Freudenthal's formula is the most classical and is still used for several computer program (e.g.\ Sage~\cite{Sage}). 
Nowadays, there exist faster algorithms. 
A possible time comparison with any of them would require an implementation on Sage, which would be unfair because of the poor computer programming skills of the authors. 

Among the mentioned faster algorithms, it is distinguished the one by W.~Baldoni and M.~Vergne~\cite{BaldoniVergne18} (see also \cite{BaldoniBeckCochetVergne06,BaldoniVergne15,Cochet05} for related results), which is based on symbolic computations of Kostant partition functions. 
See also \cite{ChristandlDoranWalter12,Schutzer12,Cavallin17} for recent different approaches. 
\end{remark}

\begin{remark}\label{remDn:referee}
This interesting remark about the behavior of $m_{\pi_{k\varepsilon_1+l\varepsilon_2}}(\mu)$ as a function on $k$ and $l$ was pointed out by the referee. 
For simplicity, we take $\mu=0$, we fix $l$ a non-negative integer, and we consider $\mathfrak g$ a classical Lie algebra of type $\tipo D_n$ for some $n\geq3$, although the general case is very similar. 
Corollary~\ref{corDn:mu=0} implies that $k\mapsto m_{\pi_{k\varepsilon_1+l\varepsilon_2}}(0)$ is a quasi-polynomial in the variable $k\geq l$ whose degree does not depend on $l$. 
In fact, its degree coincides with the degree of the polynomial $k\mapsto m_{\pi_{2k\varepsilon_1}}(0) = \binom{k+n-2}{n-2}$ (i.e.\ when $l=0$), which is equal to $n-2$. 

An interesting problem, also suggested by the referee, is to understand the behavior of the function $l\mapsto m_{\pi_{(l+h)\varepsilon_1+l\varepsilon_2}}(0)$, for some $h$ fixed. 
This does not seem to be computable from Corollary~\ref{corDn:mu=0}. 
	
\end{remark}

\begin{remark}\label{rem:applications}
In \cite[Section~7]{LR-fundstring} there is a detailed account of some applications of weight multiplicity formulas in spectral geometry (see \cite{LMR-onenorm}, \cite{BoldtLauret-onenormDirac}, \cite{Lauret-spec0cyclic}, \cite{Lauret-p-spectralens}).
These expressions for the weight multiplicities are used to determine explicitly the spectra of certain natural differential operators on a manifold (or a good orbifold) of the form $\Gamma\ba G/K$, where $G$ is a semisimple compact Lie group, $K$ is a closed subgroup of $G$ and $\Gamma$ is a finite subgroup of the maximal torus $T$ of $G$.

We next specify some cases where the formulas obtained in this article could be applied.
When $G=\Sp(n)$ and $K=\Sp(n-1)\times \Sp(1)$, the spherical representations associated to the Gelfand pair $(G,K)$ (i.e.\ the set of irreducible representations of $G$ containing non-zero vectors fixed by $K$) have highest weight of the form $k(\varepsilon_{1}+\varepsilon_{2})$ for $k\geq0$.
Consequently, Theorem~\ref{thmCn:mulip_bivar} may be applied to determine the spectrum of the Laplace--Beltrami operator acting on functions on spaces covered by the $n$-dimensional quaternionic projective space $\Sp(n)/\Sp(n-1)\times\Sp(1)$ with abelian fundamental group.

When $G=\SO(m)$ and $K=\SO(m-2)\times\SO(2)$, the corresponding spherical representations for $(G,K)$ have highest weight of the form $k\varepsilon_1+l\varepsilon_2$ for $k\geq l\geq 0$.
Thus, according to $m$ is odd or even, Theorems~\ref{thmBn:mulip_bivar} or \ref{thmDn:mulip_bivar} could be applied to the same purpose as above, for spaces covered by the $2$-Grassmannian space $G/K$ with abelian fundamental group.

In a slightly different way, we now consider $n\geq3$, $G=\SO(2n)$,  $K=\SO(2n-1)$ and more general natural differential operators.
An irreducible representation $\tau$ of $K$ induces a natural $G$-homogeneous complex vector bundle $E_\tau$ on $G/K$.
There is an associated natural differential operator $\Delta_\tau$ acting on smooth sections of $E_\tau$, which induces the differential operator $\Delta_{\tau,\Gamma}$ acting on smooth sections of $\Gamma\ba E_\tau$, that is, $\Gamma$-invariant smooth sections of $E_\tau$.
We now fix $\tau=\tau_{b\varepsilon_1}$, the irreducible representation of $K$ with highest weight $b \varepsilon_{1}$.
The corresponding $\tau_{b\varepsilon_1}$-spherical representations of $(G,K,\tau_{b\varepsilon_1})$ (i.e.\ the set of $\pi\in\widehat G$ such that $\op{Hom}_K(\tau_{b\varepsilon_1},\pi|_K)\neq0$) is equal to $\{\pi_{k\varepsilon_{1}+l\varepsilon_{2}}: k\geq b\geq l\geq 0\}$.
Consequently, Theorem~\ref{thmDn:mulip_bivar} might be used to determine the spectrum of $\Delta_{\tau_{b\varepsilon_1},\Gamma}$ for $\Gamma$ a finite subgroup of the maximal torus of $G$.
An analogous process can be done in the case $G=\SO(2n-1)$ and $K=\SO(2n-2)$.
\end{remark}

\section{Type $\tipo A_n$}\label{sec:typeAn}

Consider in $\mathfrak g=\sll(n+1,\C)$ the Cartan subalgebra $$\mathfrak h =\{\diag\big(\theta_1,\dots, \theta_{n+1}\big) : \theta_i\in\C\;\forall\, i,\; \textstyle\sum\limits_{i=1}^{n+1}\theta_i=0\}.$$
Set $\varepsilon_i\big(\diag(\theta_1,\dots,\theta_{n+1})\big)= \theta_i$ for each $1\leq i\leq n+1$.
We will use the conventions of \cite[Lecture~15]{FultonHarris-book}, that is, we correspondingly write
\begin{equation*}
 \mathfrak h^* =  \bigoplus_{i=1}^{n+1} \C\varepsilon_i / \langle \textstyle\sum\limits_{i=1}^{n+1}\varepsilon_i=0 \rangle,
\end{equation*}
and we write $\varepsilon_i$ for its image in $\mathfrak h^*$.
Consequently, the set of positive roots is given by $\Sigma^+(\mathfrak g,\mathfrak h) :=\{\varepsilon_i-\varepsilon_j: 1\leq i<j\leq n+1\}$ and the weight lattice is 
$
P(\mathfrak g):=\bigoplus_{i=1}^{n+1} \Z\varepsilon_i / \langle  \sum_{i=1}^{n+1}\varepsilon_i=0 \rangle.
$
Two weights $\mu=\sum_{i=1}^{n+1} b_i \varepsilon_i$ and $\nu=\sum_{i=1}^{n+1} c_i \varepsilon_i$ in $P(\mathfrak g)$ coincide if and only if $b_i-c_i$ is constant, independent of $i$.

A weight $\lambda =\sum_{i=1}^{n+1}a_i\varepsilon_i$ in $P(\mathfrak g)$ is dominant if and only if $a_1\geq a_2\geq \dots \geq a_{n+1}$.
By the Highest Weight Theorem, the irreducible representations of $\mathfrak g$ are in correspondence with dominant weights. 
We denote by $\pi_\lambda$ the irreducible representation with highest weight $\lambda$, which will be always written as $\lambda=\sum_{i=1}^{n+1} a_i\varepsilon_i$ with $a_{n+1}=0$. 
Thus, the irreducible representations of $G$ are in correspondence with elements in the set
$$
P^{++}(\mathfrak g) := \left\{ \textstyle\sum\limits_{i=1}^n a_i\varepsilon_i: a_i\in\Z\;\forall\, i,\; a_1\geq a_2\geq \dots \geq a_{n}\geq 0 \right\}.
$$
The fundamental weights are given by $\omega_p = \varepsilon_1+\dots+\varepsilon_p$ for each $1\leq p\leq n$. 
Thus, any integer combination between $\omega_1$ and $\omega_2$ has the form $k\varepsilon_{1}+l\varepsilon_{2}$ for some integers $k\geq l\geq0$.

For $\lambda = \sum_{i=1}^n a_i\varepsilon_i \in P^{++}(\mathfrak g)$, any weight $\mu$ of $\pi_{\lambda}$ (i.e.\ the multiplicity of $\mu$ in $\pi_{\lambda}$ is non-zero) can be written as $\mu=\sum_{i=1}^{n+1} b_i \varepsilon_i$ for some $b_1,\dots,b_{n+1}\in \N_0$ satisfying $\sum_{i=1}^{n+1} b_i=\sum_{i=1}^n a_i$.
Indeed, every weight in $\pi_{\lambda}$ is a difference between $\lambda$ and a sum of positive roots.

Let $\lambda $ and $\mu$ as in the previous paragraph. 
It is well known that (see for instance \cite[(A.19)]{FultonHarris-book}) the multiplicity of $\mu$ in $\pi_{\lambda}$ is given by the \emph{Kostka number} $K_{\lambda,\mu}$: the number of \emph{semistandard tableaux} on the Young diagram associated to $\lambda$ (i.e.\ a diagram with $a_i$ boxes in the $i$th row, with the rows of boxes lined up on the left) of type $\mu$. 
More precisely, $K_{\lambda,\mu}$ is the number of ways one can fill the boxes of the Young diagram associated to $\lambda$ with $b_1$ $1$'s, $b_2$ $2$'s, up to $b_{n+1}$ $(n+1)$'s, in such a way that the entries in each row are non-decreasing, and those in each column are strictly increasing.

The next lemma will be needed in the proof of the main result of this section.

\begin{lemma}\label{lem:virtualring(tau)An}
 Let $\mathfrak g$ be a classical Lie algebra of type $\tipo A_n$. For integers $k\geq l\geq0$ write $\tau_{k,l}=\pi_{k\varepsilon_{1}}\otimes \pi_{l\varepsilon_{1}}$.
 Then, in the virtual ring of representations, we have that $$\pi_{k\varepsilon_{1}+l\varepsilon_{2}}\simeq
 \tau_{k,l}- \tau_{k+1,l-1}.$$
\end{lemma}

\begin{proof}
The well-known fusion rule (see for instance \cite[Prop. 15.25]{FultonHarris-book})
 $$
  \tau_{k,l}=
\pi_{k\varepsilon_{1}}\otimes \pi_{l\varepsilon_{1}}\simeq
\bigoplus_{p=0}^{l} \pi_{(k+p)\varepsilon_1+(l-p)\varepsilon_2}
$$
implies
 $$\tau_{k,l}-\tau_{k+1,l-1} =
 \sum_{p=0}^{l} \pi_{(k+p)\varepsilon_1+(l-p)\varepsilon_2}
 -
 \sum_{p=1}^{l} \pi_{(k+p)\varepsilon_1+(l-p)\varepsilon_2}=
 \pi_{k\varepsilon_{1}+l\varepsilon_{2}},
 $$
and the lemma follows.
\end{proof}

We now want to calculate the weight multiplicities of the representation with highest weight a non-negative integer combination of the first two fundamental weights. 
The following multiplicity formula is probably already known, but it is included here for completeness.

\begin{theorem}\label{thmAn:multip(k,l)}\textup{(Type $\tipo A_n$)}
Let $\mathfrak g=\sll(n+1,\C)$ for some $n\geq2$ and let $k\geq l\geq 0$ integers.
Let $\mu=\sum_{i=1}^{n+1} a_i\varepsilon_i \in P(\mathfrak g)$ with $a_i\in \N_0$ for all $i$ and $\sum_{i=1}^{n+1} a_i=k+l$. If $a_i\leq k$ for all $i$, then
 \begin{align*}
m_{\pi_{k\varepsilon_1+l\varepsilon_2}}(\mu)= &  \sum_{\mathbf q\in \mathcal{Q}_{n+1}(l)} \prod_{j=1}^{l} \binom{n+1-\sum\limits_{t=0}^{j-1} \contador_t(\mu) - \sum\limits_{i=j+1}^{l} s_i^{\mathbf q}}{s_j^{\mathbf q}} \\
  &- \sum_{\mathbf q'\in \mathcal{Q}_{n+1}(l-1)} \prod_{j=1}^{l-1} \binom{n+1-\sum\limits_{t=0}^{j-1} \contador_t(\mu) - \sum\limits_{i=j+1}^{l-1} s_i^{\mathbf q'}}{s_j^{\mathbf q'}},
 \end{align*}
 and $m_{\pi_{k\varepsilon_1+l\varepsilon_2}}(\mu)=0$ otherwise, where $\contador_t(\mu) = \#\{ i: 1\leq i\leq n+1,\, a_i=t\}$,
\begin{align*}
\mathcal{Q}_{n+1}(N)&= \{ \mathbf{q}=(q_1,q_2,\dots,q_{n+1})\in \Z^{n+1}: q_1\geq q_2\geq\dots\geq q_{n+1}\geq0,\, \textstyle \sum_{i=1}^{n+1} q_i=N \},
\end{align*}
and $s_j^{\mathbf q}:= \#\{i:1\leq i\leq {n+1},\, q_i=j\}$ for $\mathbf q\in \mathcal{Q}_{n+1}(N)$ and $1\leq j\leq N$.
\end{theorem}

\begin{proof}
Since $\tau_{k,l}=\pi_{k\varepsilon_1}\otimes \pi_{l\varepsilon_1}$, 
we have that
\begin{align}\label{eq:multiptensor2An}
	m_{\tau_{k,l}}(\mu)=\sum_{\eta} m_{\pi_{k\varepsilon_{1}}}(\mu-\eta) \; m_{\pi_{l\varepsilon_{1}}}(\eta),
\end{align}
where the sum is restricted to the weights of $\pi_{l\varepsilon_{1}}$.
For $h$ any positive integer, the Young diagram associated to $\pi_{h\varepsilon_{1}}$ has only one row, of lenght $h$.
Thus, the set of weights of $\pi_{h\varepsilon_{1}}$ is given by elements of the form $\nu=\sum_{i=1}^{n+1} c_i\varepsilon_i$ with $c_1,\dots,c_{n+1}\in\N_0$ and $\sum_{i=1}^{n+1} c_i=h$, and all of them have multiplicity $1$.
Consequently, $m_{\tau_{k,l}}(\mu)$ is equal to the number of weights $\eta$ of $\pi_{l \varepsilon_{1}}$ satisfying that $\mu-\eta$ is a weight of $\pi_{k \varepsilon_{1}}$.

Let $\mathbf q\in \mathcal{Q}_{n+1}(l)$.
We want to count the number of weights $\eta=\sum_{i=1}^{n+1} b_i\varepsilon_i$ contributing to \eqref{eq:multiptensor2An} (i.e.\ $\eta$ is a weight of $\pi_{l \varepsilon_{1}}$ and $\mu-\eta$ is a weight of $\pi_{k \varepsilon_{1}}$) satisfying that $s^{\mathbf q}_j$ entries of $\eta$ are equal to $j$ for each $1\leq j\leq l$.
Clearly, $\mu-\eta$ is a weight of $\pi_{k \varepsilon_{1}}$ if and only if $a_i-b_i\geq0$ for all $1\leq i\leq n+1$.
Since for each $1\leq j\leq l$ there are $n+1-\sum_{t=0}^{j-1}\contador_t(\mu)$ $a_i$'s greater than $j-1$, then the required number is
\begin{equation}\label{eqAn:productoria}
	\binom{n+1-\sum_{t=0}^{l-1}\contador_t(\mu)}{s^{\mathbf q}_l} \binom{n+1-\sum_{t=0}^{l-2}\contador_t(\mu)-s^{\mathbf q}_l}{s^{\mathbf q}_{l-1}}\cdots \binom{n+1-\contador_0(\mu)-\sum_{j=2}^{l}s^{\mathbf q}_j}{s^{\mathbf q}_1}.
\end{equation}

We have shown that $m_{\tau_{k,l}}(\mu)$ is equal to the sum over $\mathbf q\in \mathcal{Q}_{n+1}(l)$ of  \eqref{eqAn:productoria}.
The theorem now follows by Lemma~\ref{lem:virtualring(tau)An}.
\end{proof}

We now state the closed explicit formulas for the particular cases $l=0$, $1$, and $2$. 
When $l=0$, since $\mathcal{Q}_{n+1}(0)=\{(0,\dots,0)\}$, Theorem~\ref{thmAn:multip(k,l)} immediately implies that every weight as in the hypotheses (i.e.\ $\mu=\sum_{i=1}^{n+1} b_i\varepsilon_i$ with $b_i\in\N_0$ for all $i$ and $\sum_{i=1}^{n+1} b_i=k$) has multiplicity one. 
This fact is very well known because the Young diagram associated to $\pi_{k\varepsilon_1}$ has only one row, and consequently, the number of semistandard tableaux on this diagram of type $\mu$ is one.

We now assume $l=1$. 
Let $\mu$ again as in the hypotheses of Theorem~\ref{thmAn:multip(k,l)}. 
The number of partitions of $1$ is obviously one, i.e.\ $\mathcal{Q}_{n+1}(1)=\{(1,0,\dots,0)\}$, thus $m_{\pi_{k\varepsilon_1+\varepsilon_2}} (\mu)=\binom{n+1-\ell_0(\mu)}{1}-1=n-\ell_0(\mu)$, where $\ell_0(\mu)$ is the number of zeros coordinates of $\mu$. 
It is not difficult to check that the number of semistandard tableaux of type $\mu$ is $n-\ell_0(\mu)$. 

We conclude the article stating the multiplicity formula for the irreducible representation of $\mathfrak{sl}(n+1,\C)$ with highest weight $k\varepsilon_1+2\varepsilon_2$. 
Similarly as above, the proof follows immediately from Theorem~\ref{thmAn:multip(k,l)}, since it reduces to consider the only two partitions of $2$. 
The reader may try to obtain this formula by counting semistandard tableaux of type $\mu$, and convince his/her self that the difficulty will increase for a higher $l$.


\begin{corollary}\label{corAn:for(l=2)}
Let $\mathfrak g=\sll(n+1,\C)$ for some $n\geq2$ and let $k\geq 2$ integer.
Let $\mu=\sum_{i=1}^{n+1} a_i\varepsilon_i \in P(\mathfrak g)$ with $a_i\in \N_0$ for all $i$ and $\sum_{i=1}^{n+1} a_i=k+2$. If $a_i\leq k$ for all $i$, then
 \begin{align*}
m_{\pi_{k\varepsilon_1+2\varepsilon_2}}(\mu)
  &=\binom{n+1-\contador_0(\mu)}{2}-\contador_1(\mu)
 \end{align*}
 and $m_{\pi_{k\varepsilon_1+2\varepsilon_2}}(\mu)=0$ otherwise, where $\contador_t(\mu) = \#\{ i: 1\leq i\leq n+1,\, a_i=t\}$.
\end{corollary}

We end this article with an observation pointed out by the referee, in the same spirit of Remark~\ref{remDn:referee}. 

\begin{remark}\label{remAn:referee}
We consider the `weight zero' in $\pi_{k\varepsilon_{1}+l\varepsilon_2}$, which in our convention is given by $0_{k+l}:= \sum_{i=1}^{n+1} \frac{k+l}{n+1} \varepsilon_i$. 
Clearly, $m_{\pi_{k\varepsilon_1+l\varepsilon_2}}(0_{k+l}) = 0$ unless $n+1$ divides $k+l$. 
Theorem~\ref{thmAn:multip(k,l)} does not give an explicit expression for $m_{\pi_{k\varepsilon_1+l\varepsilon_2}}(0_{k+l})$ like in Corollaries~\ref{corDn:mu=0}--\ref{corBn:mu=0}. 
However, for $l\geq0$ fixed, it implies that $m_{\pi_{k\varepsilon_1+l\varepsilon_2}}(0_{k+l})$ does not depend on $k$, for $k$ sufficiently large satisfying that $n+1$ divides $k+l$.
Moreover, the function $k\mapsto m_{\pi_{k\varepsilon_1+l\varepsilon_2}}(0_{k+l})$ is constant 
for every $k\in ln+(n+1)\N_0$.
Indeed, for such $k$, we have that $\frac{k+l}{n+1}\geq l$, thus $\ell_t(0_{k+l})=0$ for every $0\leq t\leq l-1$. 
\end{remark}

\section*{Acknowledgments}
The authors wish to thank the anonymous referee for pointing out very interesting remarks. 
This research was partially supported by grants from CONICET, FONCyT and SeCyT--UNC. The first named author was supported by the Alexander von Humboldt Foundation. 

\bibliographystyle{plain}

\end{document}